\documentclass[11pt]{article}
\usepackage{graphicx}
\usepackage{color}
\usepackage[colorlinks, linkcolor=blue,  citecolor=blue, urlcolor=blue]{hyperref}%
\usepackage{url}
\usepackage{subfig}
\usepackage{amsmath,amsthm}
\usepackage{amssymb,enumerate}
\usepackage[T1]{fontenc}
\usepackage{authblk}
\theoremstyle{plain}
\newtheorem{thm}{Theorem}[section]
\newtheorem{lem}[thm]{Lemma}

\theoremstyle{definition}

\theoremstyle{remark}
\newtheorem*{rem}{Remark}

\begin{document}

\title{On global minimizers of repulsive-attractive power-law interaction energies}

\author[1]{Jos\'e Antonio Carrillo\thanks{Author for correspondence.
E-mail: carrillo@imperial.ac.uk}}
\author[2]{Michel Chipot}
\author[1]{Yanghong Huang}

\affil[1]{Department of Mathematics, Imperial College London, London SW7 2AZ}
\affil[2]{Institut f\"ur Mathematik, Angewandte Mathematik,
Winterthurerstrasse 190, CH-8057 Z\"urich, Switzerland}


\renewcommand\Authands{ and }
\maketitle

\begin{abstract}
We consider the minimisation of power-law repulsive-attractive interaction energies which occur in many biological and physical situations. We show existence of global minimizers in the discrete setting and get bounds for their supports independently of the number of Dirac Deltas in certain range of exponents. These global discrete minimizers correspond to the stable spatial profiles of flock patterns in swarming models. Global minimizers of the continuum problem are obtained by compactness. We also illustrate our results through numerical simulations.
\end{abstract}

\section{Introduction}
Let $\mu$ be a probability measure on $\mathbb{R}^d$. We are interested in minimizing 
the interaction potential energy defined by
\begin{equation}\label{eq:contE}
E[\mu] = \frac12\int_{\mathbb{R}^d\times\mathbb{R}^d} W(x-y)
d\mu(x)d\mu(y)\,.
\end{equation}
Here, $W$ is a repulsive-attractive power-law potential
\begin{equation}\label{eq:kernelW}
W(z) = w(|z|) = \frac{|z|^\gamma}{\gamma}
-\frac{|z|^\alpha}{\alpha},\quad \gamma > \alpha \,,
\end{equation}
with the understanding that $\frac{|z|^\eta}{\eta}=\log |z|$ for $\eta=0$. Moreover, we define $W(0)=+\infty$ if $\alpha\leq 0$. This is 
the simplest possible potential that is repulsive in the short range and attractive 
in the long range. Depending on the signs of the exponents $\gamma$ and $\alpha$, 
the behaviour of the potential is depicted in Figure~\ref{fig:wfun}. Since this 
potential $W$ is bounded from below by $w(1)=\frac{1}{\gamma}-\frac{1}{\alpha}$, 
the energy $E[\mu]$ always makes sense,  with possibly positive infinite values.

\begin{figure}[thp]
\begin{center}
\includegraphics[totalheight=0.175\textheight]{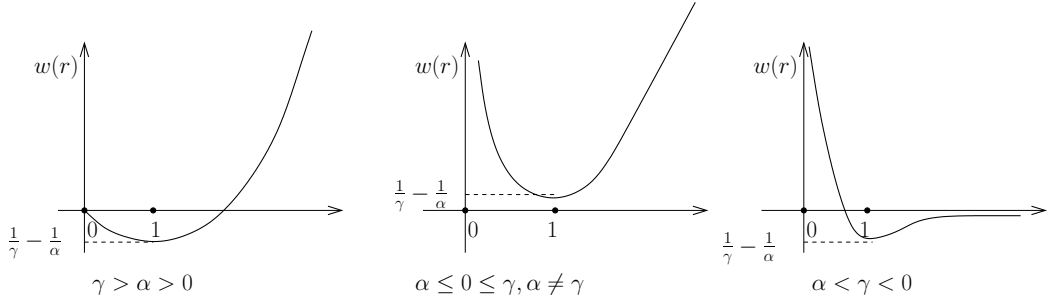}
\end{center}
\caption{Three different behaviours of $w(r) = \frac{r^\gamma}{\gamma}
-\frac{r^\alpha}{\alpha}$, $\gamma>\alpha$.}
\label{fig:wfun}
\end{figure}

The minimizers of the energy $E[\mu]$ are related to  
stationary states for the aggregation equation $\rho_t = \nabla\cdot(\rho \nabla W*\rho)$ 
studied in~\cite{Carrillo-McCann-Villani03,
Carrillo-McCann-Villani06,BCL,CDFLS,BLL} with 
repulsive-attractive potentials~\cite{FellnerRaoul1,FellnerRaoul2,FHK,FH,Raoul,BCLR,BCLR2}. 
The set of local minimizers of the interaction energy, in both the discrete setting of 
empirical measures (equal mass Dirac Deltas) and the continuum setting of general probability 
measures,  can exhibit rich complicated structure as studied numerically in \cite{KSUB,BCLR2}. 
In fact, it is shown in \cite{BCLR2} that the dimensionality of the support of local minimizers 
of~\eqref{eq:contE} depends on the strength of the repulsion at zero of the potential $W$. In other 
words, as the repulsion at zero gets stronger (i.e., $\alpha$ gets smaller), the support of local 
minimizers gets larger in Hausdorff dimension.

From the viewpoint of applications, these models with nonlocal interactions are ubiquitous 
in the literature. Convex attractive potentials appear in granular  media~\cite{BenedettoCagliotiCarrilloPulvirenti98,LT,Carrillo-McCann-Villani03,Carrillo-McCann-Villani06}.  More sophisticated potentials like~\eqref{eq:kernelW} are included to take into account short  range repulsion and long range attraction in kinematic models of collective behaviour of animals, see \cite{Mogilner2,Dorsogna,Mogilner2003,KSUB,KCBFL} and the references therein. The minimization of the interaction energy in the discrete settings is of paramount importance for the structure of virus capsides \cite{Viral_Capside}, for self-assembly materials in chemical engineering design \cite{Wales1995,Wales2010,Rechtsman2010}, and for flock patterns in animal swarms \cite{BUKB,soccerball,CHM}.

Despite the efforts in understanding the qualitative behaviour of stationary solutions to the aggregation equation $\rho_t = \nabla\cdot(\rho \nabla W*\rho)$ and the structure of local minimizers of the interaction energy $E[\mu]$, there are no general results addressing the global minimization of $E[\mu]$ in the natural framework of probability measures. 
See \cite{CFT} for a recent analysis of this question in the more restricted set of bounded or binary densities. Here, we will first try to find solutions in the restricted set of atomic measures. 

The interest of understanding the global discrete minimizers of the interaction energy is not purely mathematical. The discrete global minimizers will give the spatial profile of typical flocking patterns obtained in simplified models for social interaction between individuals as in \cite{Albietal,KSUB} based on the famous 3-zones models, see for instance \cite{Huth:Wissel,lukeman}. Moreover, due to the recent nonlinear stability results in \cite{CHM}, we know now that the stability properties of the discrete global minimizer as stationary solution of the first order ODE model
$$
\dot{x}_i=-\sum_{j\neq i}^n \nabla
W\left(x_i-x_j\right), \quad i=1,\dots, n\,,
$$
lead to stability properties of the flock profiles for the second order model in swarming introduced in \cite{Dorsogna} or with additional alignment mechanisms as the Cucker-Smale interaction \cite{CS1,CS2}, see also \cite{Albietal} and the discussion therein.

Our objective is to show the existence of global minimizers of the interaction energy defined 
on probability measures under some conditions on the exponents. Our approach starts with 
the discrete setting by showing qualitative properties about the global minimizers in the 
set of equal mass Dirac Deltas. These discrete approximations are used extensively in the literature
to show various properties of the minimizers~\cite{Dorsogna,FHK,KSUB,BUKB}, but the existence
as well as the convergence of these discrete minimizers is not justified in general. In a certain range of exponents, 
we will prove that the diameter of the support of discrete minimizers does not depend on the 
number of Dirac Deltas. This result together with standard 
compactness arguments will result in our desired global minimizers among probability measures.

In fact, our strategy to show the confinement of discrete minimizers is in the same spirit as 
the proof of confinement of solutions of the aggregation equation in~\cite{CDFLS2,BCY}. In 
our case, the ideas of the proof in Section 2 will be based on convexity-type arguments 
in the range of exponents $\gamma>\alpha\geq 1$ to show the uniform bound in the diameter of 
global minimizers in the discrete setting. Section 3 will be devoted to more refined results in 
one dimension. We show that for very repulsive potentials, the bounds on the diameter is not uniform
in the number of Dirac Deltas, complemented by numerical simulations; in the range of exponents
$\gamma>1>\alpha$, the minimizers turn out to be unique (up to translation), analogous to the simplified 
displacement convexity in 1D; in the special case $\gamma=2$ and $\alpha=1$, we can find the 
minimizers and show the convergence to the continuous minimizer explicitly.


\section{Existence of Global minimizers}

We will first consider the discrete setting where $\mu$ is a convex combinations of Dirac Deltas, i.e.,
\[
\mu = \frac{1}{n} \sum_{i=1}^n \delta_{x_i},\quad x_i \in \mathbb{R}^d.
\]
Setting
\begin{equation}\label{eq:sumEng}
E_n(x_1,\cdots,x_n)=\sum_{i\neq j}^n \left(\frac{|x_i-x_j|^\gamma}{\gamma}-\frac{|x_i-x_j|^\alpha}{\alpha}\right)\,,
\end{equation}
for such a $\mu$ one has
$
E[\mu] = \frac{1}{2n^2} E_n(x_1,\cdots,x_n)\,.
$
In the definition of the energy, we can include the self-interaction for non singular cases, $\alpha>0$, since both definitions coincide. Fixing $W(0)=+\infty$ for singular kernels makes $W$ upper semi-continuous, and the self-interaction must be excluded to have finite energy configurations.

Let us remark that due to translational invariance of the interaction energy, minimizers of the 
interaction energy $E[\mu]$ can only be expected up to translations. Moreover, when the 
potential is radially symmetric, as in our case, then any isometry in $\mathbb{R}^d$ will also 
leave invariant the interaction energy. These invariances are also inherited by the discrete 
counterpart $E_n(x_1,\cdots,x_n)$. We will first consider the minimizers of $E_n(x)$ among 
all $x=(x_1,\cdots,x_n)\in (\mathbb{R}^d)^n$, and then the convergence to the global minimizers of $E[\mu]$
as $n$ goes to infinity.

\subsection{Existence of minimizer: Discrete setting}
Let us consider for $\alpha <\gamma$, the derivative of the radial potential
\[
    w'(r) = r^{\gamma-1}-r^{\alpha-1} =
    r^{\alpha-1}\left(r^{\gamma-\alpha}-1\right)\,,
\]
which obviously vanishes for $r=1$ and for $r=0$ when $\alpha>1$. We conclude from the 
sign of derivatives, that $w(r)$ attains always a global minimum at $r=1$. There are, following 
the values of $\alpha<\gamma$, three types of behaviours for $w$ that are shown in 
Figure~\ref{fig:wfun}. In all the three cases, $E_n$ is bounded from below since
\[
    E_n(x) \geq n^2\left(\frac{1}{\gamma}-\frac{1}{\alpha}\right)\,,
\]
with the understanding that $\frac{|x|^\eta}{\eta}=\log |x|$ for $\eta=0$. We set
\begin{equation}\label{eq:Indefn}
    I_n = \inf_{x \in (\mathbb{R}^d)^n} E_n(x).
\end{equation}
Using the translational invariance of $E_n(x_1,\cdots,x_n)$, we can assume without 
loss of generality that $x_1=0$ what we do along this subsection. First we have the following lemma
showing that $I_n$ is achieved, which can be proved by discussing different ranges of the 
exponents $\gamma$ and $\alpha$.

\begin{lem} For any finite $n(\geq 2)$, the minimum value $I_n$ is obtained for
  some discrete minimizers in $(\mathbb{R}^d)^n$ .
\end{lem}

\vskip 6pt

\underline{The case $0<\alpha<\gamma$.} We claim that
\begin{equation}\label{eq:Case1Bound}
   n^2\left( \frac{1}{\gamma}-\frac{1}{\alpha} \right)
    \leq I_n <0.
\end{equation}
Indeed consider $x = (x_1,\cdots,x_n) \in (\mathbb{R}^d)^n$ such
that $x_1,\cdots,x_n$ are aligned and $|x_{i}-x_{i+1}|=\frac{1}{n}$.
Then for any $i,j$ one has
$
    0 < |x_i-x_j|\leq 1
$
and $w(|x_i-x_j|)<0$. Therefore~\eqref{eq:Case1Bound} follows.

Let us show that the infimum $I_n$ is achieved.
Let $x\in(\mathbb{R}^d)^n$. Set $R = \max_{i,j} |x_i-x_j|$. A 
minimizer is sought among the points such that $E_n(x)<0$ and one
has for such a point 
\[
    \frac{R^\gamma}{\gamma} 
    \leq \sum_{i,j} \frac{|x_i-x_j|^\gamma}{\gamma}
    < \sum_{i,j} \frac{|x_i-x_j|^\alpha}{\alpha}
    \leq n^2 \frac{R^\alpha}{\alpha}.
\]
This implies the upper bound 
\begin{equation}\label{eq:case1R}
    R \leq \left(\frac{n^2\gamma}{\alpha}\right)^{\frac{1}{\gamma-\alpha}}.
\end{equation}
Thus, since $x_1=0$, all the  $x_i$'s have to be in the ball
of center $0$ and radius
$\left(n^2\gamma/\alpha\right)^{\frac{1}{\gamma-\alpha}}$, i.e.,
$x$ has to be in a compact set of $(\mathbb{R}^d)^n$. Since $E_n(x)$ is 
continuous, the infimum $I_n$ is achieved. Note that the bound on the radius, where
all Dirac Deltas are contained, depends a priori on $n$.

\vskip 6pt

\underline{The case $\alpha \leq 0 \leq \gamma$ and $\alpha\neq\gamma$.} In this 
case $w(0^+)=+\infty$ and $w(\infty)=+\infty$. We minimize among all $x$ such that 
$x_i\neq x_j$ for $i\neq j$.  Note that $w$ and $I_n$ are both positive. Since $w(r)\to +\infty$
    as $r\to 0$ or $r\to \infty$, there exists $a_n,b_n>0$ such that
\begin{equation*}
a_n < 1<b_n,\quad w(a_n)=w(b_n)>I_n.
\end{equation*}
Let $x\in (\mathbb{R}^d)^n$. If for a couple $i,j$, one has
\begin{equation}\label{eq:Case2bound}
|x_i-x_j| <a_n \quad \mbox{or}\quad |x_i-x_j|>b_n
\end{equation}
then one has 
$
    E_n(x)>I_n.
$
Thus the infimum~\eqref{eq:Indefn} will not be achieved among the points $x$
satisfying~\eqref{eq:Case2bound} but among those in
\[
    \{ x\in (\mathbb{R}^d)^n\mid a_n \leq |x_i-x_j|\leq b_n
\}
\]
Since the set above is compact, being closed and contained in $(B(0,b_n))^n$
due to $x_1=0$, the infimum $I_n$ is achieved.

\vskip 6pt

\underline{The case $\alpha<\gamma<0$.} In this case $I_n <0$. Indeed it is
enough to choose 
\[
    |x_i-x_j|>1\qquad \forall i,j
\]
to get $E_n(x)<0$. Since $w(0^+)=+\infty$, we 
minimize $E_n$ among the points $x$ such that $x_i\neq x_j$, $i\neq j$.
Thus the summation is on $n^2-n$ couples $(i,j)$. Denote by
$
    x^k = (x_1^k,\cdots,x_n^k) \in (\mathbb{R}^d)^n
$
a minimizing sequence of $E_n$. Since $w(r)\to +\infty$
as $r\to 0$, there exists a number $a_n<1$ such that
\[
    w(a_n) > n(n-1)\left(\frac{1}{\alpha}-\frac{1}{\gamma}\right)>0.
\]
If for a couple $(i,j)$ one has $|x_i^k-x_j^k|<a_n$ then
\[
    E(x^k) > w(a_n) 
    + ({n^2-n-1})\left(\frac{1}{\gamma}-\frac{1}{\alpha}\right)
    >\left(\frac{1}{\alpha}-\frac{1}{\gamma}\right)>0
\]
and $x^k$ cannot be a minimizing sequence. So without loss of generality, we may assume that
$|x_i^k-x_j^k|\geq a_n,\, \forall i,j$. Let us denote by $y_1,\cdots,y_d$ the coordinates 
in $\mathbb{R}^d$. Without loss of generality, we can assume by relabelling and isometry invariance 
that for every $k$ one has
\[ 
    x_1^k =0,\quad x_{i}^k \in \{ y=(y_1,\cdots,y_d)\mid y_d\geq 0\}\,.
\]
Suppose that
$
    x_i^k = (x^k_{i,1},\cdots,x^k_{i,d})
$
and the numbering of the points is done in such a way that
$
    x_{i,d}^k\leq x_{i+1,d}^k.
$

We next claim that one can assume that $x_{i+1,d}^k-x_{i,d}^k \leq 1$, $\forall i$.
Indeed if not, let $i_0$ be the first index such that
\[
    x_{i_0+1,d}^k-x_{i_0,d}^k >1.
\]
Let us leave the first $x_i^k$ until $i_0$ unchanged and replace for
$i>i_0$ , $x_i^k$ by
\[
    \tilde{x}_i^k = x_i^k - \{x_{i_0+1,d}^k-x_{i_0,d}^k-1\}e_d
\]
where $e_d$ is the $d$-vector of the the canonical basis of
$\mathbb{R}^d$, 
i.e., we shift $x_i^k$ down in the direction $e_d$ by
$x_{i_0+1,d}^k-x_{i_0,d}^k-1$. Denote by $\tilde{x}_i^k$ the new
sequence obtained in this manner. One has
\begin{align*}
w(|\tilde{x}_i^k-\tilde{x}_j^k|) 
&< w(|x_i^k-x_j^k|),\quad \forall i\leq i_0<j, \\
w(|\tilde{x}_i^k-\tilde{x}_j^k|) 
&= w(|x_i^k-x_j^k|),\quad \forall i_0< i<j, \mbox{or } i<j\leq i_0
\end{align*}
and thus one has obtained a minimizing sequence with
\[
    \tilde{x}_{i_0+1,d}^k - \tilde{x}_{i_0,d}^k\leq 1,\qquad \forall i
\]
i.e., $0\leq x_{i,d}^k\leq n-1$, for all $i$.

Repeating this process in the other directions one can assume without loss
of generality that 
\begin{equation}\label{eq:case3R}
 x_i^k \in [0,n-1]^d
\end{equation}
for all $k$, i.e., that $x^k$ is in a compact subset of $(\mathbb{R}^d)^n$,
and extracting a convergent subsequence, we obtain our desired minimizer in $[0,n-1]^d$.

\subsection{Existence of minimizer: General measures}
The estimates~\eqref{eq:case1R} and~\eqref{eq:case3R} give estimates on the support 
of a minimizer of~\eqref{eq:Indefn}. However, these estimates depend on $n$. 
We will show now that the diameter of any minimizer of~\eqref{eq:Indefn} can sometimes
be bounded independently of $n$. 

\begin{thm}\label{thm:bndxn}
Suppose that $1\leq \alpha < \gamma$. Then the diameter of any global minimizer 
of $E_n$ achieving the infimum in \eqref{eq:Indefn} is bounded independently of $n$.
\end{thm}

\begin{proof}
At a point $x=(x_1,\cdots,x_n) \in (\mathbb{R}^d)^n$ where the minimum of $E_n$ is achieved one has
\[
    0=\nabla_{x_k}E_n(x_1,\cdots,x_n)=\nabla_{x_k} \sum_{j\neq k} 
    \left(\frac{|x_{k}-x_j|^\gamma}{\gamma}-
    \frac{|x_k-x_j|^\alpha}{\alpha}\right),\quad k=1,2,\cdots,n.
\]
Since $\nabla_x \big(|x|^\eta/\eta\big) = |x|^{\eta-2}x$, we obtain
\begin{equation}\label{jjj}
    \sum_{j\neq k} |x_k-x_j|^{\gamma-2}(x_k-x_j)=\sum_{j\neq k}
    |x_k-x_j|^{\alpha-2}(x_k-x_j),\quad k=1,\cdots,n.
\end{equation}
Suppose the points are labelled in such a way that
\[
    |x_n-x_1|=\max_{i,j} |x_i-x_j|.
\]
Then for $k=1$ and $n$ in~\eqref{jjj}, we get
\begin{align*}
    \sum_{j\neq 1} |x_1-x_j|^{\gamma-2}(x_1-x_j)&=\sum_{j\neq 1}
    |x_1-x_j|^{\alpha-2}(x_1-x_j),\\
    \sum_{j\neq n} |x_n-x_j|^{\gamma-2}(x_n-x_j)&=\sum_{j\neq n}
    |x_n-x_j|^{\alpha-2}(x_n-x_j).
\end{align*}
By subtraction, this leads to
\begin{multline*}
    \sum_{j\neq 1,n} \Big( |x_n-x_j|^{\gamma-2}(x_n-x_j)-
    |x_1-x_j|^{\gamma-2}(x_1-x_j)\Big)
    +2|x_n-x_1|^{\gamma-2}(x_n-x_1) \\
= \sum_{j\neq n}|x_n-x_j|^{\alpha-2}(x_n-x_j)
-\sum_{j\neq 1}|x_n-x_j|^{\alpha-2}(x_1-x_j).
\end{multline*}
Taking the scalar product of both sides with $x_n-x_1$ we obtain
\begin{multline*}
    \sum_{j\neq 1,n} \Big( |x_n-x_j|^{\gamma-2}(x_n-x_j)-
    |x_1-x_j|^{\gamma-2}(x_1-x_j),x_n-x_1\Big)
    +2|x_n-x_1|^{\gamma} \\
= \sum_{j\neq n}|x_n-x_j|^{\alpha-2}(x_n-x_j,x_n-x_1)
-\sum_{j\neq 1}|x_n-x_j|^{\alpha-2}(x_1-x_j,x_n-x_1).
\end{multline*}

For $\gamma \geq 2$, there exists a constant $C_\gamma>0$ such that (see \cite{MC})
\begin{equation}\label{eq:gammaineq1}
    \big( |\eta|^{\gamma-2}\eta-|\xi|^{\gamma-2}\xi,\eta-\xi\big)
    \geq C_\gamma |\eta-\xi|^\gamma,\qquad 
    \forall \eta,\xi\in \mathbb{R}^d.
\end{equation}
Note that this is nothing else that the modulus of convexity (in the sense of~\cite{Carrillo-McCann-Villani06}) of the potential $|x|^\gamma$.
Thus estimating from above, we derive
\begin{align*}
\big( (n-2)C_\gamma+2\big)|x_n-x_1|^\gamma 
&\leq \sum_{j\neq n} |x_n-x_j|^{\alpha-1}|x_n-x_1|+
\sum_{j\neq 1}|x_1-x_j|^{\alpha-1}|x_n-x_1| \cr
&\leq 2(n-1)|x_n-x_1|^\alpha.
\end{align*}
Thus if $a\wedge b$  denotes the minimum of two numbers $a$ and $b$, we
derive
\[
    \big( C_\gamma\wedge 1\big) n |x_n-x_1|^\gamma 
    \leq 2(n-1)|x_n-x_1|^\alpha.
\]
That is
\[
    |x_n-x_1| \leq \left(\frac{2}{C_\gamma\wedge
    1}\frac{n-1}{n}\right)^{\frac{1}{\gamma-\alpha}}
    \leq \left(\frac{2}{C_\gamma\wedge 1}\right)^{\frac{1}{\gamma-\alpha}},
\]
which proves the theorem in the case $\gamma \geq 2$. In the case where
$1<\gamma<2$, one can replace~\eqref{eq:gammaineq1} with
\begin{equation*}
    \big( |\eta|^{\gamma-2}\eta-|\xi|^{\gamma-2}\xi,\eta-\xi\big)
    \geq c_\gamma\big\{|\eta|+|\xi|\big\}^{\gamma-2} |\eta-\xi|^2,\qquad 
    \forall \eta,\xi\in \mathbb{R}^d, 
\end{equation*}
for some constant $c_\gamma$, (see \cite{MC}). We  get, arguing as above,
\[
\sum_{j\neq 1,n} c_\gamma\big\{|x_n-x_j|+|x_1-x_j|\big\}^{\gamma}
+2|x_n-x_1|^\gamma
\leq 2(n-1)|x_n-x_1|^\alpha.
\]
No since $\gamma-2<0$, $|x_n-x_j|<|x_n-x_1|$ and $|x_1-x_j|
<|x_n-x_1|$, we derive that
\[
    \big( (n-2)c_\gamma 2^{\gamma-2}+2\big)|x_n-x_1|^\gamma 
    \leq 2(n-1)|x_n-x_1|^\alpha.
\]
We thus obtain the bound
\[
    |x_n-x_1| < \left\{ \frac{2}{(2^{\gamma-2}c_\gamma)\wedge
1}\right\}^{\frac{1}{\gamma-\alpha}},
\]
which completes the proof of the theorem.
\end{proof}

As a direct consequence of this bound independent of the number of Dirac Deltas, 
we can prove  existence of global minimizers in the continuous setting.

\begin{thm} Suppose that $1\leq \alpha <\gamma$. Then global minimizers associated to 
the global minimum of $E_n(x)$ with zero center of mass converge as $n\to\infty$ toward 
a global minimizer among all probability measures with bounded moments of order $\gamma$ 
of the interaction energy $E[\mu]$ in \eqref{eq:contE} .
\end{thm}
\begin{proof}
    Let $x^n\in(\mathbb{R}^d)^n$ be a minimizer of~\eqref{eq:sumEng} and 
\[
    \mu^n = \frac{1}{n}\sum_j^n \delta_{x^n_j}
\]
be the associated discrete measure. From Theorem~\ref{thm:bndxn}, the
radius of the supports of the measures $\mu^n$ are bounded uniformly in $n$ by $R$, provided
that the center of the mass $\int_{\mathbb{R}^d} xd\mu^n$ is normalized to be the origin.
By Prokhorov's theorem~\cite{Prokhorov}, $\{\mu^n\}$ is compact in the weak-$*$ topology of measures and also in the metric space induced by $\gamma$-Wasserstein distance $d_\gamma$ between probability measures (see \cite{GS,V} for definition and basic properties). Then there is a measured $\mu^*$  supported on $B(0,R)$ such that 
\[
    \mu^n \rightharpoonup \mu^* \qquad \mbox{and} \qquad d_\gamma(\mu^n,\mu^*)\to 0 
\]
as $n$ goes to infinity. Note that the notion of convergence of a sequence of probability measures 
in $d_\gamma$ is equivalent to weak convergence of the measures plus convergence of 
the moments of order $\gamma$, see \cite[Chapter 9]{V}. Let $\nu$ be any probability measure 
on $\mathbb{R}^d$ with bounded moment of order $\gamma$, then $E[\nu]<\infty$. Moreover, there 
is a sequence of discrete measures $\nu^n$
of the form
\[
 \nu^n = \frac{1}{n} \sum_{j=1}^n \delta_{y_j^n},
\]
such that $d_\gamma(\nu^n,\nu)\to 0$, and thus $\nu^n \rightharpoonup \nu$, see \cite{GS,V}. 
By the definition of $E_n$ in \eqref{eq:Indefn}, we deduce
\[
 E[\nu^n]= \frac12\int_{\mathbb{R}^d\times \mathbb{R}^d} W(x-y)d\nu^n(x)d\nu^n(y) \geq  
 \frac12\int_{\mathbb{R}^d\times \mathbb{R}^d} W(x-y)d\mu^n(x)d\mu^n(y) =
E[\mu^n].
\]
On the other hand, since
\[
    d_\gamma(\mu^n \otimes \mu^n , \mu^*\otimes \mu^*)\to 0, \qquad
    d_\gamma(\nu^n \otimes \nu^n , \nu\otimes \nu)\to 0\,,
\]
as $n\to\infty$, and the function $w(x-y) = |x-y|^\gamma/\gamma-|x-y|^\alpha/\alpha$ is
Lipschitz continuous on bounded sets in $\mathbb{R}^d\times \mathbb{R}^d$ with growth of order $\gamma$ at infinity, then 
\begin{align*}
E[\mu^*] &= \frac12\int_{\mathbb{R}^d\times \mathbb{R}^d} W(x-y)d\mu^*(x)d\mu^*(y) =  \lim_{n\to\infty}  
\frac12\int_{\mathbb{R}^d\times \mathbb{R}^d} W(x-y)d\mu^n(x)d\mu^n(y) \cr
&\leq  \lim_{n\to\infty} \frac12\int_{\mathbb{R}^d\times \mathbb{R}^d} W(x-y)d\nu^n(x)d\nu^n(y) = E[\nu].
\end{align*}
Therefore, $\mu^*$ must be a global minimizer of $E[\mu]$ in the set of probability measures with bounded moments of order $\gamma$.
\end{proof}

\begin{rem}
Global minimizers of the energy in the continuum setting might be a convex combination of a finite number of Dirac Deltas. Numerical experiments suggest that it is always the case in the range $2<\alpha<\gamma$. It is an open problem in this range to show that global minimizers in the discrete case do not change (except symmetries) for $n$ large enough and coincide with global minimizers of the continuum setting.
\end{rem}


\section{Further Remarks in one dimension}

In this section, we concentrate on the one dimensional case ($d=1$) for more refined properties. 

\subsection{Confinement of Discrete Global Minimizers}
We check first how sharp are the conditions on the exponents of the potential to get the confinement of global discrete minimizers independently of $n$.
In fact, when the potential is very repulsive at the origin, we can show that a uniform bound in $n$ of the diameter of global minimizers of the discrete setting does not hold. If $x$ is a minimizer of $E_n(x)$, we will always 
assume that the labelling of $x_i$s is in increasing order:
$
    x_1\leq x_2 \cdots \leq x_n.
$

\begin{thm}
    Suppose $\alpha<\gamma<0$ and $\alpha<-2$. If $x$ is a minimizer of $E_n$, then
    there exists a constant $C_{\alpha,\gamma}$ such that for $n$ large
    enough 
\[
    x_n-x_1 \geq C_{\alpha,\gamma}\, n^{1+\frac{2}{\alpha}}
\]
holds.
\end{thm}

\begin{proof} Set $C=\frac{1}{\alpha}-\frac{1}{\gamma}>0$. Denote by $a_n$
    the unique element of $\mathbb{R}$ such that 
\begin{equation}\label{eq:wacn}
w(a_n)=Cn^2.
\end{equation}
If $x$ is a
    minimizer of $E_n$, we claim that 
\begin{equation}\label{eq:1ddist}
    x_{i+1}-x_i \geq a_n,\qquad \forall i=1,\dots,n.
\end{equation}
Indeed otherwise,
\[
E_n(x) \geq w(x_{i+1}-x_i)-(n^2-n-1)C >
Cn^2 - (n^2-n-1)C=(n+1)C>0
\]
and we know that in this case $E_n<0$.
From~\eqref{eq:wacn} we derive
\[
    Cn^2 = \frac{ {a_n}^\gamma}{\gamma}-
    \frac{ {a_n}^\alpha}{\alpha}=
    -\frac{ {a_n}^\alpha}{\alpha}\left(
    1-\frac{\alpha}{\gamma}{a_n}^{\gamma-\alpha}
    \right)\geq -\frac{ {a_n}^\alpha}{2\alpha},
\]
for $n$ large enough (recall that $a_n\to 0$ when $n\to\infty$). It follows
that
\[
    a_n \geq \big(-2\alpha C n^2\big)^{\frac{1}{\alpha}}
    =\big(-2\alpha C\big)^{\frac{1}{\alpha}}n^{\frac{2}{\alpha}}.
\]
Combining this with~\eqref{eq:1ddist} we get
\[
    x_n-x_1 \geq (n-1)a_n \geq \frac{n}{2}\big(-2\alpha
    C\big)^{\frac{1}{\alpha}} n^{\frac{2}{\alpha}}
    =\frac{(-2\alpha C)^{\frac{1}{\alpha}}}{2}n^{1+\frac{2}{\alpha}}
\]
for $n$ large enough, proving the desired estimate with $C_{\alpha,\gamma}
=(-2\alpha C)^{\frac{1}{\alpha}}/2$.
\end{proof}

\begin{figure}[htp]
\begin{center}
\subfloat[$\alpha=-2.5$]{\includegraphics[keepaspectratio=true,
width=.43\textwidth]{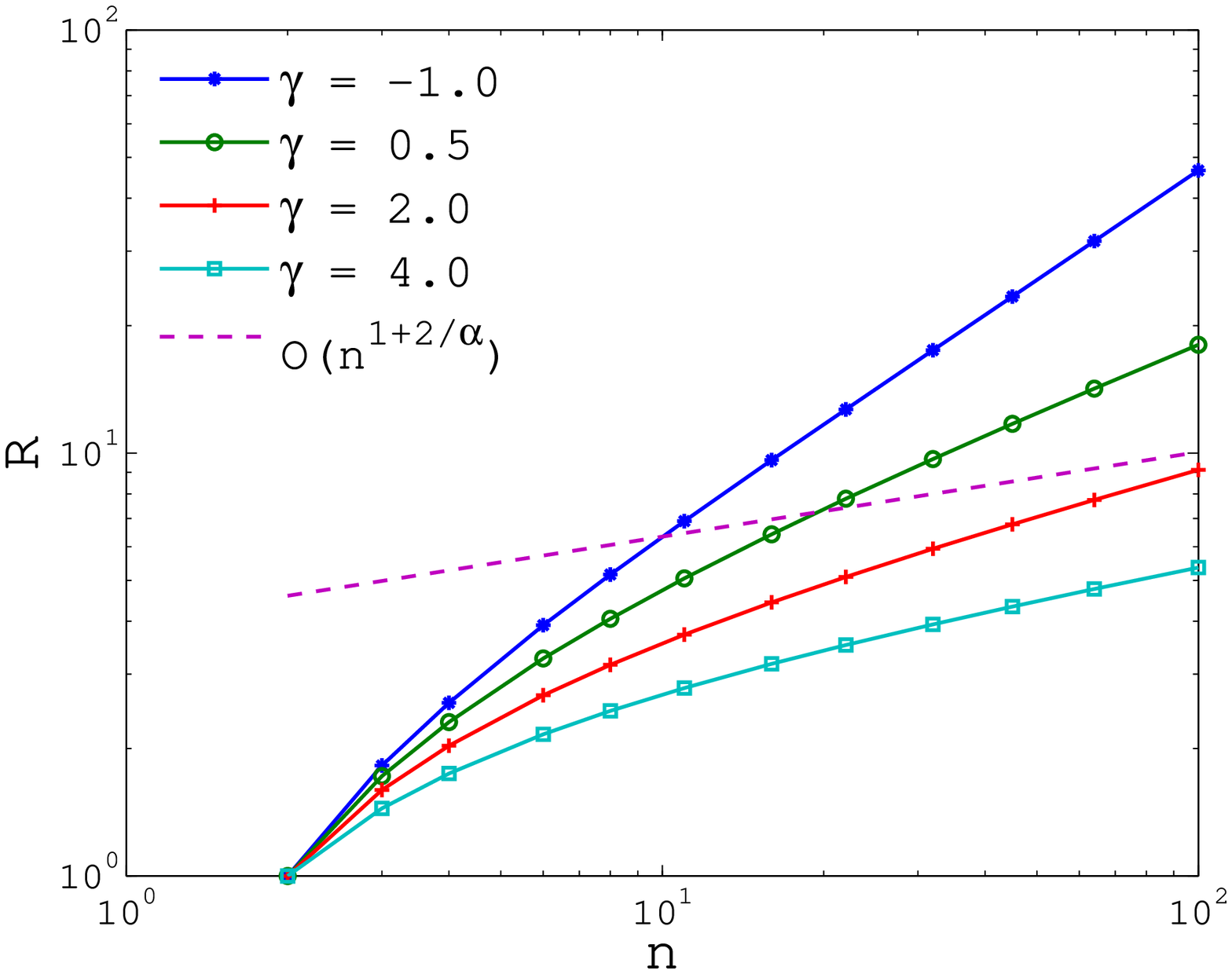}}
$~~$ 
\subfloat[$\alpha=-1.5$]{\includegraphics[keepaspectratio=true,
width=.43\textwidth]{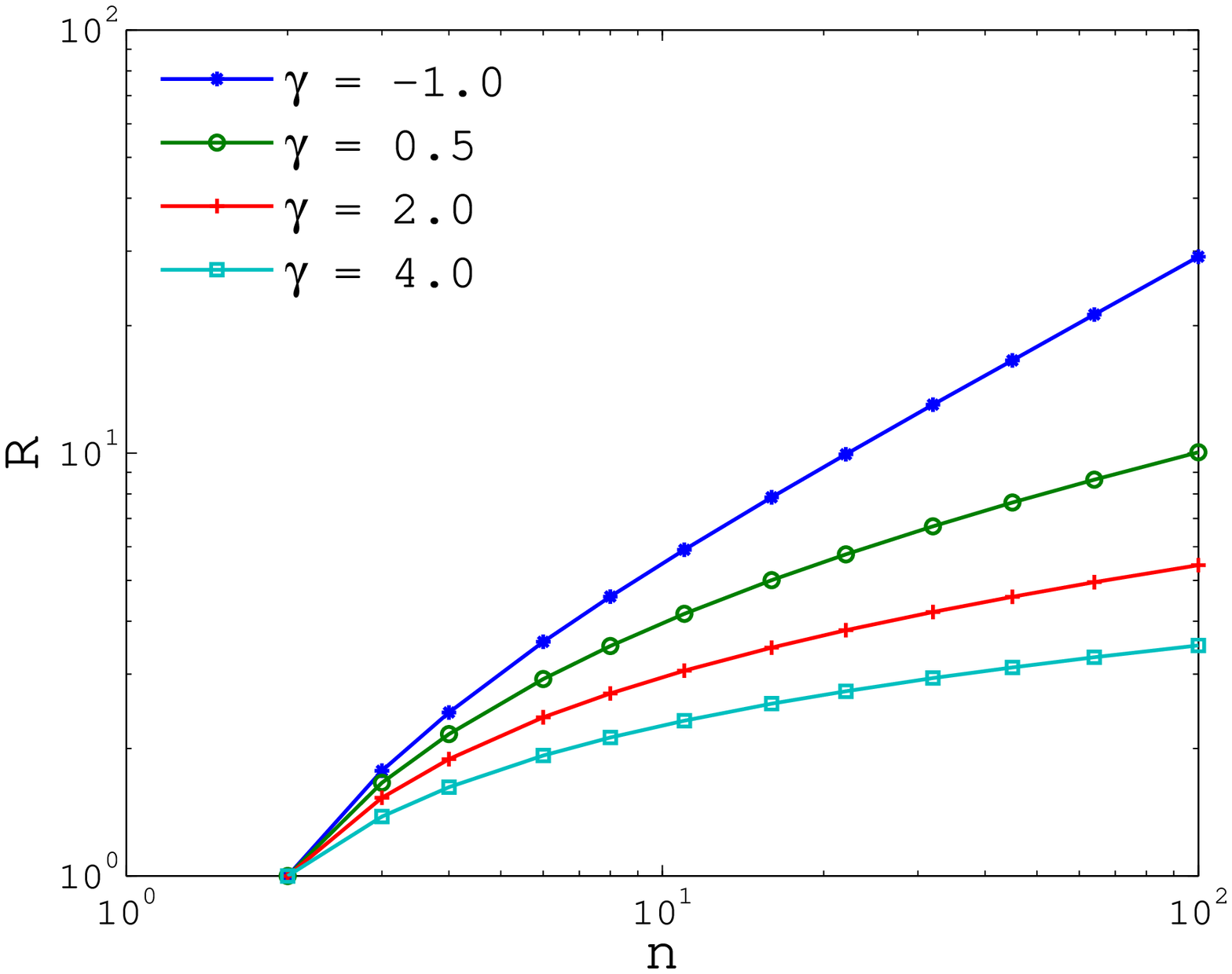}}
    \end{center}
\caption{The dependence of the diameter $R = \max_{i,j}|x_i-x_j|$ on the number of 
particles $n$.}
\label{fig:radius}
\end{figure}

This property for the minimizers of this very repulsive case is similar to H-stability in 
statistical mechanics~\cite{HStable}, where the minimal distance between 
two particles is expected to be constant when $n$ is large, and crystallization occurs. 
This also suggests that the lower bound $O(n^{1+\frac{2}{\alpha}})$ is not sharp, which is verified in Figure~\ref{fig:radius}. 

In fact, numerical experiments in \cite{BCLR2,BCY} suggest that confinement happens for $-1<\alpha <1$. It is an open problem to get a uniform bound in the support of the discrete minimizers as in Section 2 in this range. In the range $\alpha \leq -1$, our numerical simulations suggest that spreading of the support happens for all $\gamma$ with a decreasing spreading rate as $\gamma$ increases.

\subsection{Uniqueness of global minimizers}
We turn now to the issue of uniqueness (up to isometry) of global discrete and 
continuum minimizers. If $x$ is a minimizer of $E_n(x)$, we can always assume 
at the expense of a translation that the center of mass is zero, that 
is $\frac{x_1+\cdots+x_n}{n}=0$. Let us recall that
\[
    E_n(x) = \sum_{i\neq j}^n w(|x_i-x_j|)
\]
with the convention that $x_i\neq x_j$ when $i\neq j$, $\alpha <0$. 

\begin{lem}\label{lem:unique1d}
    Suppose that $\alpha \leq 1$, $\gamma\geq 1$, and $\alpha < \gamma$. Let $x,y$ be two points of
    $\mathbb{R}^n$ such that 
\begin{subequations}\label{eq:1dminxy}
\begin{equation}\label{eq:1dminx}
\frac{x_1+\cdots+x_n}{n}=0,\ x_1\leq x_2 \leq \cdots \leq x_n,
\end{equation}
\begin{equation}\label{eq:1dminy}
\frac{y_1+\cdots+y_n}{n}=0,\ y_1\leq y_2 \leq \cdots \leq y_n. 
\end{equation}
\end{subequations}
then 
\[
    E_n\left(\frac{x+y}{2}\right)< \frac{E_n(x)+E_n(y)}{2}
\]
unless $x=y$.
\end{lem}
\begin{proof}
One has $w''(r)=(\gamma-1)r^{\gamma-2}-(\alpha-1)r^{\alpha-2}>0$, for all $r>0$.
Thus  $w$ is strictly convex. Then one has by the strict
convexity of $w$,
\begin{align*}
E_n\left(\frac{x+y}{2}\right) 
&= \sum_{i\neq j}^n w\left(\Big|\frac{x_i+y_i}{2}-\frac{x_j+y_j}{2}\Big|\right) \cr 
&= \sum_{i\neq j}^n w\left(\Big|\frac{x_i-x_j}{2}+\frac{y_i-y_j}{2}\Big|\right) \cr 
&\leq \frac{1}{2}\left(
\sum_{i\neq j}^n w(|x_i-x_j|)+\sum_{i\neq j}^n w(|y_i-y_j|) 
\right) =\frac{E_n(x)+E_n(y)}{2}
\end{align*}
The equality above is strict unless $x_i-x_j=y_i-y_j$ for all $i,j$, that is
$x_{i+1}-x_i=y_{i+1}-y_i$. Therefore $x=y$.
\end{proof}

As a consequence, we can now state the following result regarding the uniqueness of global discrete minimizers.

\begin{thm}\label{thm:unique}
    Suppose $\alpha \leq 1$, $\gamma\geq 1$, and $\alpha < \gamma$. Up to translations, the minimizer $x$
    of $E_n$ is unique and symmetric with respect to its center of mass.
\end{thm}
\begin{proof} Let $x$, $y$ be two minimizers of $E_n$
    satisfying~\eqref{eq:1dminxy}. If $x\neq y$, by Lemma~\ref{lem:unique1d},
    one has
\[
    E_n\left(\frac{x+y}{2}\right)<\frac{E_n(x)+E_n(y)}{2}=I_n
\]
and a contradiction. This shows the uniqueness of a minimizer satisfying~\eqref{eq:1dminx}.
Denote now by $s$ the symmetry defined by 
$
    s(\xi) = -\xi$, $\xi \in \mathbb{R}.
$
If $x$ is a minimizer of $E_n(x)$ satisfying~\eqref{eq:1dminx} then $y$ defined
by
\[
    y_i=s(x_{n+1-i})\qquad i=1,\cdots,n
\]
is also a minimizer satisfying~\eqref{eq:1dminy}. Thus by uniqueness
\[
    x_i=-x_{n+1-i}\qquad i=1,\cdots,n,
\]
and this completes the proof of the theorem.
\end{proof}

\begin{rem}[Uniqueness and displacement convexity in one dimension]
Lemma \ref{lem:unique1d} and Theorem \ref{thm:unique} are just discrete versions 
of uniqueness results for the continuum interaction functional~\eqref{eq:contE}. 
In the seminal work of R. McCann \cite{Mc} that introduces the notion of displacement 
convexity, he already dealt with the uniqueness (up to translation) of the interaction 
energy functional \eqref{eq:contE} using the theory of optimal transportation: 
if $W$ is strictly convex in $\mathbb{R}^d$, then the global minimizer is unique 
among probability measures by fixing the center of mass, as the energy $E[\mu]$ is 
(strictly) displacement convex. However,  the displacement convexity of a functional 
is less strict in one dimension than that in higher dimensions. As proven in~\cite{CFP}, 
to check the displacement convexity of the energy $E[\mu]$ in one dimension, it is 
enough to check the convexity of the function $w(r)$ for $r>0$. Therefore, if $w(r)$ 
is strictly convex in $(0,\infty)$, then the energy functional \eqref{eq:contE} is strictly 
displacement convex for probability measures with zero center of mass. As a consequence, 
the global minimizer of \eqref{eq:contE} in the set of probability measures is unique up 
to translations. Lemma \ref{lem:unique1d} shows that this condition is equivalent 
to $\alpha \leq 1$, $\gamma\geq 1$, and $\alpha < \gamma$, for power-law potentials. 
Finally, the convexity of $E_n$ in Lemma~\ref{lem:unique1d} is just the displacement
 convexity of the energy functional \eqref{eq:contE} restricted to discrete measures. 
We included the proofs of the convexity and uniqueness because they are quite straightforward 
in this case, without appealing to more involved concepts in optimal transportation.
\end{rem}

\begin{rem}[Explicit convergence to uniform density] As a final example we consider the case where $\gamma=2$,
 $\alpha=1$, which corresponds to quadratic attraction and Newtonian repulsion in 
one dimension (see \cite{FellnerRaoul1}). When $x$ is a minimizer of $E_n(x)$, we have by \eqref{jjj} that
\[
    \sum_{j\neq k} (x_k-x_j) = 
    \sum_{j\neq k} \mbox{sign}(x_k-x_j) = k-n-1,\quad \forall\ k=1,\cdots,n.
\]
Replacing the index $k$ by $k+1$, the equation becomes
\[
\sum_{j\neq {k+1}} (x_{k+1}-x_j) = 
\sum_{j\neq {k+1}} \mbox{sign}(x_{k+1}-x_j) = 2k-n+1,\quad \forall\
k=1,\cdots,n-1.
\]
Subtracting the two equations above, we get
\[
    n(x_{k+1}-x_k)=2,\ \forall\ k=1,\cdots,n-1,
\]
that is $x_{k+1}-x_k=\frac{2}{n}$, for all $k=1,\cdots,n-1$.

This shows that in the case $\gamma=2$ and $\alpha=1$, the points $x_i$ are uniformly 
distributed;  as $n$ goes to infinity, the corresponding discrete measure 
$\mu=\frac{1}{n}\sum_{i=1}^n \delta_{x_i}$ converges to the uniform probability 
measure on the interval $[-1,1]$. This uniform density is known to be the global 
minimizer of the energy $E[\mu]$ in the continuum setting, see \cite{FellnerRaoul1,CFT}.
\end{rem}

\section*{Acknowledgement}
JAC acknowledges support from projects MTM2011-27739-C04-02,
2009-SGR-345 from Ag\`encia de Gesti\'o d'Ajuts Universitaris i de Recerca-Generalitat 
de Catalunya, and the Royal Society through a Wolfson Research Merit Award. JAC and YH 
acknowledges support from the Engineering and Physical Sciences Research Council (UK) grant 
number EP/K008404/1. MC acknowledges the support of a J. Nelder fellowship from Imperial 
College where this work was initiated. The research leading to these results has received 
funding from Lithuanian-Swiss cooperation programme to reduce economic and social disparities within the enlarged European Union under project agreement No CH-3-SMM-01/0.  The research of MC was also supported by  the Swiss National Science Foundation under the contract $\#$  200021-146620. Finally, the paper was completed  during a visit of MC at the Newton Institute in Cambridge. The nice atmosphere of the institute is greatly acknowledged.


\begin{thebibliography}{10}

\bibitem{Albietal}
G.~Albi, D.~Balagu\'e, J.~A. Carrillo, J.~von Brecht. Stability Analysis of Flock and Mill rings for 2nd Order Models in Swarming. To appear in {\em SIAM J. Appl. Math.}

\bibitem{BCLR}
D.~Balagu\'e, J.~A. Carrillo, T.~Laurent, and G.~Raoul.
Nonlocal interactions by repulsive-attractive potentials: Radial
ins/stability. {\em Physica D}, 260:5-25, 2013.

\bibitem{BCLR2}
D.~Balagu\'e, J.~A. Carrillo, T.~Laurent, and G.~Raoul.
 Dimensionality of local minimizers of the interaction energy.
 {\em Arch. Rat. Mech. Anal.}, 209(3):1055--1088, 2013.

\bibitem{BCY}
D.~Balagu{\'e}, J.~A. Carrillo, and Y.~Yao.
Confinement for repulsive-attractive kernels.
To appear in {\em DCDS-A}.

\bibitem{BenedettoCagliotiCarrilloPulvirenti98}
D. Benedetto, E. Caglioti, J.A. Carrillo, M. Pulvirenti.
A non-maxwellian steady distribution for one-dimensional
granular media. {\em J. Stat. Phys.}, 91:979--990, 1998.

\bibitem{BCL}
A.~Bertozzi, J.~A. Carrillo, and T.~Laurent.
 Blowup in multidimensional aggregation equations with mildly singular
  interaction kernels.
 {\em Nonlinearity}, 22:683--710, 2009.

\bibitem{BLL}
A.~L. Bertozzi, T.~Laurent, and F.~L\'{e}ger.
 Aggregation and spreading via the newtonian potential: the dynamics
  of patch solutions.
 {\em Math. Models Methods Appl. Sci.},
  22(supp01):1140005, 2012.

\bibitem{CDFLS}
J.~A. Carrillo, M.~Di~Francesco, A.~Figalli, T.~Laurent, and D.~Slep\v{c}ev.
 Global-in-time weak measure solutions and finite-time aggregation for
  nonlocal interaction equations.
 {\em Duke Math. J.}, 156:229--271, 2011.

\bibitem{CDFLS2}
J.~A. Carrillo, M.~Di~Francesco, A.~Figalli, T.~Laurent, and D.~Slep{\v{c}}ev.
 Confinement in nonlocal interaction equations.
 {\em Nonlinear Anal.}, 75(2):550--558, 2012.

\bibitem{CFP}
J.~A. Carrillo, L.~C.~F. Ferreira, J.~C. Precioso.
 A mass-transportation approach to a one dimensional fluid mechanics model with nonlocal velocity.
 {\em Adv. Math.}, 231(1):306--327, 2012. 

\bibitem{CHM} J.~A. Carrillo, Y. Huang, S. Martin, Nonlinear stability of flock solutions in second-order swarming models, {\em Nonlinear Analysis: Real World Applications}, 17:332--343, 2014.

\bibitem{Carrillo-McCann-Villani03}
J.~A. Carrillo, R.~J. McCann, and C.~Villani.
 Kinetic equilibration rates for granular media and related equations:
 entropy dissipation and mass transportation estimates.
 {\em Rev. Mat. Iberoamericana}, 19(3):971--1018, 2003.

\bibitem{Carrillo-McCann-Villani06}
J.~A. Carrillo, R.~J. McCann, and C.~Villani.
 Contractions in the $2$-wasserstein length space and thermalization
  of granular media.
 {\em Arch. Rat. Mech. Anal.}, 179:217--263, 2006.

\bibitem{MC} M. Chipot: {\em  Elliptic Equations: An Introductory Course}, Birkh\"{a}user Advanced Text, 2009.

\bibitem{CFT}
R. Choksi, R. Fetecau, I. Topaloglu. On Minimizers of Interaction Functionals with Competing Attractive and Repulsive Potentials. Preprint 2013.

\bibitem{CS1}
F.~Cucker and S.~Smale.
Emergent behavior in flocks.
{\em IEEE Trans. Automat. Control}, 52(5):852--862, 2007.

\bibitem{CS2}
F.~Cucker and S.~Smale.
On the mathematics of emergence.
{\em Jpn. J. Math.}, 2(1):197--227, 2007.

\bibitem{Dorsogna}
M.~R. D'Orsogna, Y.~Chuang, A.~Bertozzi, and L.~Chayes.
 Self-propelled particles with soft-core interactions: patterns,
  stability and collapse.
 {\em Phys. Rev. Lett.}, 96(104302), 2006.

\bibitem{Wales1995}
J.~P.~K. Doye, D.~J. Wales, and R.~S. Berry.
 The effect of the range of the potential on the structures of
  clusters.
 {\em J. Chem. Phys.}, 103:4234--4249, 1995.

\bibitem{FellnerRaoul1}
K.~Fellner and G.~Raoul.
 Stable stationary states of non-local interaction equations.
 {\em Math. Models Methods Appl. Sci.}, 20(12):2267--2291, 2010.

\bibitem{FellnerRaoul2}
K.~Fellner and G.~Raoul.
 Stability of stationary states of non-local equations with singular
  interaction potentials.
 {\em Math. Comput. Modelling}, 53(7-8):1436--1450, 2011.

\bibitem{FHK}
R.~C. Fetecau, Y.~Huang, and T.~Kolokolnikov.
 Swarm dynamics and equilibria for a nonlocal aggregation model.
 {\em Nonlinearity}, 24(10):2681--2716, 2011.

\bibitem{FH}
R.~C. Fetecau and Y.~Huang.
 Equilibria of biological aggregations with nonlocal repulsive--attractive interactions.
 {\em Physica D}, 260:49--64, 2013.

\bibitem{GS}
C.~R. Givens and R.~M. Shortt.
 A class of {W}asserstein metrics for probability distributions.
 {\em Michigan Math. J.}, 31(2):231--240, 1984.

\bibitem{Viral_Capside}
M.~F. Hagan and D.~Chandler.
 Dynamic pathways for viral capsid assembly.
 {\em Biophysical Journal}, 91:42--54, 2006.

\bibitem{Huth:Wissel}
A.~Huth and C.~Wissel.
The simulation of fish schools in comparison with experimental data.
{\em Ecol. Model.}, 75/76:135--145, 1994.

\bibitem{KCBFL}
T. Kolokolnikov, J. A. Carrillo, A. Bertozzi, R. Fetecau, M. Lewis.
 Emergent behaviour in multi-particle systems with non-local interactions,
 {\em Physica D: Nonlinear Phenomena}, 260:1--4, 2013.

\bibitem{KSUB}
T.~Kolokonikov, H.~Sun, D.~Uminsky, and A.~Bertozzi.
 Stability of ring patterns arising from 2d particle interactions.
 {\em Physical Review E}, 84(1):015203, 2011.

\bibitem{LT}
H.~Li and G.~Toscani. 
Long--time asymptotics of kinetic models of granular flows.
{\em Arch. Rat. Mech. Anal.}, 172(3):407--428, 2004.

\bibitem{lukeman}
R.~Lukeman, Y.X. Li, and L.~Edelstein-Keshet.
Inferring individual rules from collective behavior.
{\em Proc. Natl. Acad. Sci. U.S.A.}, 107(28):12576--12580, 2010.

\bibitem{Mc}
R.J. McCann, {\em A convexity principle for interacting
gases}, Adv. Math., 128, (1997), 153-179.

\bibitem{Mogilner2}
A.~Mogilner and L.~Edelstein-Keshet.
 A non-local model for a swarm.
 {\em J. Math. Bio.}, 38:534--570, 1999.

\bibitem{Mogilner2003}
A.~Mogilner, L.~Edelstein-Keshet, L.~Bent, and A.~Spiros.
 Mutual interactions, potentials, and individual distance in a social
  aggregation.
 {\em J. Math. Biol.}, 47(4):353--389, 2003.

\bibitem{Prokhorov}
Y.~V.~Prokhorov.
Convergence of random processes and limit theorems in probability theory.
{\em Theory Probab. Appl.} (in English translation), 1(2):157--214, 1956.

\bibitem{Raoul}
G.~Raoul.
 Non-local interaction equations: Stationary states and stability
  analysis.
 {\em Differential Integral Equations}, 25(5-6):417--440, 2012.

\bibitem{Rechtsman2010}
M. C.~Rechtsman, F. H.~Stillinger, and S.~Torquato.
 Optimized interactions for targeted self-assembly: application to a
  honeycomb lattice.
 {\em Phys. Rev. Lett.}, 95(22), 2005.

\bibitem{HStable}
D.~Ruelle.
{\em Statistical mechanics: {R}igorous results}
 W. A. Benjamin, Inc., New York-Amsterdam, 1969

\bibitem{V}
C.~Villani.
 {\em Topics in optimal transportation}, volume~58 of {\em Graduate
  Studies in Mathematics}.
 American Mathematical Society, Providence, RI, 2003.

\bibitem{soccerball}
J.~von Brecht and D.~Uminsky.
 On soccer balls and linearized inverse statistical mechanics.
 {\em J. Nonlinear Sci.}, 22(6):935--959, 2012.

\bibitem{BUKB}
J.~von Brecht, D.~Uminsky, T.~Kolokolnikov, and A.~Bertozzi.
 Predicting pattern formation in particle interactions.
 {\em Math. Mod. Meth. Appl. Sci.}, 22:1140002, 2012.

\bibitem{Wales2010}
D.~J. Wales.
 Energy landscapes of clusters bound by short-ranged potentials.
 {\em Chem. Eur. J. Chem. Phys.}, 11:2491--2494, 2010.

\end{thebibliography}
\end{document}